\documentclass[preprint]{imsart}

\RequirePackage[OT1]{fontenc}
\RequirePackage{amsthm,amsmath,amssymb}
\RequirePackage[numbers]{natbib}
\RequirePackage[colorlinks,citecolor=blue,urlcolor=blue]{hyperref}
\usepackage{stmaryrd}
\usepackage[sans]{dsfont}

\newcommand{\trans}{\mathsf{T}}
\newcommand{\e} {\varepsilon}

\newcommand{\Z} {\mathbb{Z}}
\newcommand{\N} {\mathbb{N}}
\newcommand{\E} {\mathbb{E}}
\newcommand{\R} {\mathbb{R}}
\newcommand{\prob} {\mathbb{P}}

\newcommand{\di}[1]{\operatorname{d}\!#1}

\newcommand{\Q} {\mathcal{Q}}

\newcommand{\K} {\mathcal{K}}
\newcommand{\B} {\mathcal{B}}

\newcommand{\F} {\mathcal{F}}

\startlocaldefs
\numberwithin{equation}{section}
\theoremstyle{plain}

\newtheorem{theorem}{Theorem}[section]
\newtheorem{proposition}{Proposition}[section]
\newtheorem{remark}{Remark}[section]
\endlocaldefs

\begin{document}

\begin{frontmatter}
\title{A Lyapunov view on positive Harris recurrence of multiclass queueing networks}

\runtitle{Lyapunov view on positive Harris recurrence}


\begin{aug}
\author{\fnms{Michael} \snm{Sch\"onlein}\ead[label=e1]{schoenlein@mathematik.uni-wuerzburg.de}}

\runauthor{M. Sch\"onlein}

\affiliation{University of W\"urzburg}

\address{Institute for Mathematics\\
         University of W\"urzburg\\
         Emil-Fischer Stra\ss e 40\\
         97074 W\"urzburg, Germany\\
\printead{e1}\\
\phantom{E-mail:\ }
}
\end{aug}

\begin{abstract}
This paper addresses the question when the underlying Markov process of a multiclass queueing network is  positive Harris recurrent. It is well-known that stability of the fluid limit model is a sufficient condition for this. Hence, stability of fluid (limit) models is of vital interest. Recently, it has been shown that if the fluid model satisfies certain properties, it is stable if and only if there exists a Lyapunov function. In this paper a new method is provided to conclude that the underlying Markov process is positive Harris recurrent if the fluid model is stable by using explicitly the Lyapunov function of the fluid model.
\end{abstract}



\end{frontmatter}

\section{Introduction}\label{intro}

Typically a multiclass queueing network consists of objects, for instance jobs or customers, that are waiting for service in buffers in front of diverse stations. After service completion, a served job either moves to another buffer at some further station or leaves the network. 

A multiclass queueing network consists of $J$ service stations and $K$ classes of customers. The \emph{interarrival times} for customers of class $k \in \{1,...K\}$ are given by positive random variables 
$ a_k(n) ,  \mbox{ with } n =1,2,3,...\,$, and the of class $k$ customers are given by positive random variables $ s_k(n) , \mbox{ with }  n =1,2,3...\, $. Each customer class is exclusively served at a certain station. The many-to-one mapping $c: \{1,...,K\} \rightarrow \{1,...,J\}$ determines which customer class is served at which station. The corresponding $J \times K$ matrix $C$, called the \emph{constituency matrix}, is defined by $C_{jk} := 1$ if $c(k)=j$ and $c_{jk}=0$ else. For station $j \in \{1,...,J\}$, the set $C(j) := \{ k \in \{1,...,K\}: c(k) =j \}$ is the collection of all customer classes that are served at station $j$.
After a class $k$ customer received service at the station $c(k)$ its routing is given by a $K$ dimensional \emph{Bernoulli random variable} $\phi^k$. 
To be precise, each component of $\phi^k(n)$ is either $0$ or $1$, but the entry $1$ appears at most once. Let $e_k$ denote the $k$th standard basis vector for $\R^K$. Then, the $n$th served class~$k$ customer at station $c(k)$ becomes a class~$l$ customer after service completion if $\phi^k(n) = e_l$, and the customer leaves the network  if $\phi^k(n) = 0$.
For some customer class $k$ the interarrival time may be $a_k(n)=\infty$ for all $n$. Then, the exogenous arrival process is null. The corresponding notation is the following $ \mathcal{E}:= \{k \in \{1,...,K\} : a_k(n) < \infty, \, n \geq1\}$. Further, the buffer at each station is assumed to have infinite capacity.
Throughout this paper we pose the following general assumptions on the interarrival and the service times. 
\begin{enumerate}
\item[$(A1)$] The sequences $a_1,...,a_K, s_1,...,s_K$ and $\phi^1,...,\phi^K$ 
are identically and independently distributed and mutually independent.
\item[$(A2)$] The first moments satisfy
\begin{equation*}
\begin{split}
\alpha_k&:=\E[\,a_k(1)\,]^{-1} < \infty  \text{ for } \,k \in \mathcal{E}, \\ 
\mu_k&:=\E[\,s_k(1)\,]^{-1} < \infty  \text{ for }\,k \in \{1,...,K\},\\
P_k&:=\E[\,\phi^k(1)\,] \geq 0 \quad \, \,  \text{ for }\,k \in \{1,...,K\},
\end{split}
\end{equation*}
and the spectral radius of the matrix $P=( P_1\, ... \, P_K)$ is strictly less than one.
\item[$(A3)$] The distributions of the interarrival times are unbounded and spread out.
\end{enumerate}
A distribution $\nu$ of the interarrival times is called \emph{unbounded}\index{distribution!unbounded} if for each class $k\in \mathcal{E}$ and for all $t\geq 0$ it holds that
\begin{align*}
\prob_{\nu}[a_k(1) \geq t] := \int_{t}^\infty a_k(1) \nu(\di{s})>0.
\end{align*}
Unboundedness expresses that arbitrarily large interarrival times appear with positive probability. Moreover, the distribution of the interarrival times of  customer class $k \in \mathcal{E}$ is said to \emph{spread out}\index{distribution!spread out} if there exists some $l_k \in \{1,2,3,...\}$ and some nonnegative function $q_k:\R_+ \rightarrow \R_+$ with $\int_0^{\infty} q_k(s) \di{s} >0$ such that for all $0 \leq a <b$,
\begin{align*}
\prob_{\nu} \left[ \sum_{i=1}^{l_k} a_k(i) \in [a,b] \right] \,\, \geq \, \,\int_a^b q_k(s)\di{s}.
\end{align*}
We refer to the triple $(a,s,\phi)$ as the \emph{primitive increments} of the multiclass queueing network. 

\subsubsection*{State Space and underlying Markov process}

We restrict ourselves to \emph{head-of-the-line (HL)} queueing networks. The term HL indicates that within each class the customers are served in \emph{First-In-First-Out} (FIFO) order. The evolution of the  HL multiclass queueing network will be described by a stochastic process $X=\{X(t),\, t\geq 0\}$. The corresponding state space is denoted by $(\vartimes, \mathcal{B})$, where $\B$ denotes the Borel $\sigma$-algebra of $\vartimes$. The precise description depends on the particular discipline. We will not introduce the full state space here, the interested reader is referred to \cite{bramsonLN,dai}. In general the state space is of the form
\begin{align*}
 \vartimes = \Big\{x\, :\, x =\big( (k,w),u,v,z\big) \in (\Z\times \R)^\infty \times \R^{|\mathcal{E}|} \times \R^K \times \R^K\Big\}.
\end{align*}
Here $(\Z\times \R)^\infty$ denotes the set of finitely terminated sequences taking values in $\Z\times \R$. The meaning of each of its components are the following. 

The global order of the customers in the network is given by the pair of sequences $(k,w)= \big( (k_1,w_1), (k_2,w_2), ... ,(k_l,w_l) \big) \in (\Z\times \R)^\infty$. The first entry $k_i \in \{1,...,K\}$ denotes the current class of the $i$th customer and the second entry $w_i\geq 0$ reflects the elapsed time since the customer $i$ entered the class $k_i$. The order of elements $(k_i,w_i)$ in the sequence $(k,w)$ is defined by the property that it is descending in $w_i$. The number of customers of each class is denoted by $q=(q_1,\ldots,q_K)$ and $\|q\|:= \sum_{k=1}^K q_k$ is the total number of customers in the state.
The variable $u \in \R^{|\mathcal{E}|}$ denotes the \emph{residual interarrival time}, where $u_k>0$ denotes the remaining time before the next arrival of a class $k \in \mathcal{E}$ customer from outside the network.
The component $v \in \R^K$ represents the \emph{residual service time}, i.e. the coordinate $v_k$ denotes the remaining time service time for the oldest class $k$ customer, where $v_k\geq 0$ and $v_k=0$ only if $q_k=0$.
The component $z_k$ of  $z \in [0,1{]}^{K}$ denotes the proportion of the service effort of station $c(k)$ that the oldest class $k$ customer receives, while other class $k$ customers do not receive any service. This represents the HL property. For each station $j$ we have that if $\sum_{k\in C(j)} q_k >0$, it holds that $ \sum_{k\in C(j)} z_k= 1$, where $z_k=0$ if $q_k=0$. If station $j$ is empty, i.e. $\sum_{k\in C(j)} q_k =0$, then  $\sum_{k\in C(j)} z_k =0$. 

The state space $(\vartimes, \B)$ is measurable. For the objective of the present paper it suffices to consider $\big( q,u,v\big)$ of the state space. With a slight abuse of notation we call $|x|: = \|q\| + \|u\| + \|v\|$ a norm on $\vartimes$. Further, let $\vartimes$ be equipped with the natural induced topology then $\{x \in \vartimes\, :\, |x|\leq \kappa\}$ is a compact subset of $\vartimes$ for every $\kappa>0$. 

The underlying stochastic process $X$ is piecewise deterministic. In particular, $X$ is a Borel right process and satisfies the strong Markov property, cf. \cite{bramsonLN}, \cite{dai}. Thus, the process is equipped with the basic ingredients defining a strong Markov process. So, the process  $X$ is  Borel right and is defined on a measurable space $(\Omega, \F)$ with values in the measurable space $(\vartimes,\B)$. Furthermore, the process is adapted to a filtration $\{ \F_t \}$ and $\{\prob_{x}, \, x \in \vartimes\}$ are probability measures on $(\Omega,\F)$ such that for every $ x \in \vartimes$ we have $\prob_x[X(0)=x]=1$. The collection 
\begin{align*}
 \Big( \Omega, \F, \{\F_t\}, \{X(t),t\geq0\}, \{\prob_x,x\in \vartimes \}\Big)
\end{align*}
defines a strong Markov family. The transition function $P$ is defined by $P(t,x,A) = \prob_x[X(t) \in A]$ for all $t\geq 0$, $x \in \vartimes$, and $A \in \B$.

For a set $A \in \B(E)$ let $ \tau_A := \inf \{ t\geq 0 \, :\, X(t) \in A\}$ denote the \emph{first entrance time}, and for $\delta >0$ the \emph{first entrance time past} $\delta$ is given by $\tau_A (\delta):= \inf \{ t\geq \delta \, :\, X(t) \in A\}$. By the D\'{e}but Theorem, cf. \cite[Theorem~A~5.1]{sharpe1988general}, the first entrance time defines a stopping time. Furthermore, for $A\in \B(E)$ we consider the \emph{occupation time} $\eta_A$, describing the number of visits by $X$ to $A$, given by 
\begin{align*}
\eta_A:= \int_0^{\infty} \mathds{1}_{\{X(t) \in A\}} 
 \di{t}.
\end{align*}
A Markov process $X$ is called \emph{Harris recurrent} if there exists a nontrivial $\sigma$-finite measure $\nu$ such that whenever $\nu(A)>0$ and $A \in \B(E)$ it holds that 
$
 \prob_x\left[ \eta_A = \infty \right] = 1
$
for all $x\in E$.  If the invariant measure can be normalized to a probability measure, the Markov process $X$ is called \emph{positive Harris recurrent}. 
A multiclass queueing network is called \emph{stable} if the underlying Markov process, denoted by $X$, is positive Harris recurrent.

Beginning from the early 1990s with the works of Rybko and Stolyar \cite{rybkostolyar}, Stolyar \cite{stolyar95} and Dai \cite{dai} an effective approach to investigate the stability of multiclass queueing networks is to consider the limits of scaled versions of the underlying Markov process.
To this end, let $(r_n,x_n)_{n \in \N}$ be a sequence of pairs, where $r_n \in \R_+$ and $x_n \in \vartimes$ is a sequence of initial states. We assume that the sequence of pairs satisfies the following conditions
\begin{align}\label{cond-scaling-sequence}
\lim_{n \rightarrow \infty} r_n = \infty, \quad \limsup_{n \rightarrow \infty} \tfrac{\|q_n\|}{r_n} < \infty, \quad \lim_{n \rightarrow \infty} \tfrac{\|u_n\|}{r_n} =\lim_{n \rightarrow \infty} \tfrac{\|v_n\|}{r_n}=0,
\end{align}
where $q_n, u_n$ and $v_n$ denote the queue length, the residual interarrival time, and the residual service time, respectively.  In the sequel, we consider the family $X':= \{X_n(t):= \tfrac{1}{r_n} X^{x_n}(r_n\,t): t\geq 0 ,n \in \N\} $ of Markov processes, where the superscript $x_n$ expresses the dependence on the initial state $x_n \in \vartimes$. Recall that, for each HL queueing network, $(r_n,x_n)_{n\in \N}$ satisfying \eqref{cond-scaling-sequence}, and $\omega \in G$,  there is a subsequence of pairs $(r_{n_i},x_{n_i})_{i\in \N}$ such that
\begin{align}\label{limit-scaled-qn}
  \lim_{i\rightarrow \infty} \tfrac{1}{r_{n_i}} X^{x_{n_i}}(r_{n_i}t,\omega) =
\overline{X}(t,\omega) 
\end{align}
uniformly on compact sets (u.o.c.), that is, for each $t\geq0$ it holds that
\begin{align*}
  \lim_{i\rightarrow \infty} \sup_{0\leq s\leq t}
 | \tfrac{1}{r_{n_i}} X^{x_{n_i}}(r_{n_i}s,\omega) -
\overline{X}(s,\omega)| =0. 
\end{align*}
The set of all these limits is called the fluid limit model, denoted by $\mathcal{FLM}$. A fluid limit model of a queueing discipline is said to be \emph{stable} if there is a $\tau > 0$ such that for any fluid limit $\overline{X}(\cdot)$ the $\overline{Q}(\cdot)$ component satisfies $\overline{Q}(t) = 0$ for all $t \geq \tau\|\overline{Q}(0)\|$.


\medskip

\textsc{Theorem} (Dai \cite{dai}). \hspace{.4cm}{\it  Let a queueing discipline be fixed. Assume that Assumptions (A1)-(A3) are satisfied. If the fluid limit model is stable, then the queueing network is stable.}

\medskip

Furthermore, the work also shows that the fluid limits satisfy a set of equations that is based on the first moments of the primitive increments. A completely deterministic analog, called the associated fluid network, is obtained as the set of solutions to these equations. The associated fluid network is a superset to the fluid limit model and, thus, the stability analysis of multiclass queueing networks can be undertaken deterministicly.

This result released a series of works establishing characterizing stability of the associated fluid networks. cf. \cite{bramsonHLPPS}, \cite{chen}, \cite{yechenpw},  \cite{chenpriority}, \cite{daischool}, \cite{daihaba07}, \cite{yechen}. In most cases the approach to verify the proposed stability conditions is by means of Lyapunov arguments. In \cite{QS-2012} a consistent Lyapunov theory has been developed for fluid networks where the set of fluid level processes satisfies the following conditions: (a) Lipschitz continuity with respect to a global Lipschitz constant (b) scaling invariance (c) shift invariance (d) closedness in the topology of uniform convergence (e) a concatenation property holds (f) the set of fluid level processes depends lower semicontinuously on the initial fluid level. It is shown that stability of a fluid network satisfying (a)-(f) is equivalent to the existence of a continuous Lyapunov function.

In this paper we make use of the converse Lyapunov statement by explicitely using the Lyapunov function admitted by the stable fluid network to construct a Foster-Lyapunov function for the underlying Markov process. To this end, we prove a modified version of the Foster-Lyapunov theorem in \cite[Proposition~4.5]{bramsonLN}.  The line of argument is similar to/inspired by the one for Theorem~2 in \cite{foss-overview}.

\begin{theorem}\label{thm:strict-dai}
Let a queueing discipline be fixed and assume that the Assumptions (A1)-(A3) hold. Suppose that the discipline is so that the associated fluid network satisfies $(a)-(f)$. If the associated fluid network is stable, then the queueing network is stable.
\end{theorem}

The remainder of the paper is structured as follows. In Section~2 we briefly recall known properties of fluid networks and display an equivalent Lyapunov characterization of stability of fluid networks satisfying $(a)-(f)$. Section~3 is devoted to the proof of Theorem~\ref{thm:strict-dai}, where we explicitly use the Lyapunov function admitted by the stable associated fluid network to establish a Foster-Lyapunov function for the underlying Markov process. For this reason, we first prove an appropriate version of the Foster-Lyapunov theorem to conclude that the positive Harris recurrence of the underlying Markov process.

\section{Preliminaries on fluid networks}\label{sec:fl-stability}

In this section we provide an approach that was first considered by Rybko and Stolyar in 1992 \cite{rybkostolyar} and was further developed by Stolyar \cite{stolyar95} and Dai \cite{dai}. 

It is well-known that any fluid limit satisfies the following set of dynamic equations, where $M=\operatorname{diag}(\mu)$:
\begin{align}
\label{T4.bal}
{Q}(t) &= {Q}(0) +  \alpha t + P^\trans \, M  {T}(t) - M  {T}(t) \geq 0,\\ \label{T4.2}
{T}(0) &=0 \text{  and  } {T}(\cdot) \text{ is nondecreasing}, \\\label{T4.3}
{W}(t) &= C\,M^{-1}\,({Q}(0) + {A}(t)) - C\, {T}(t),\\ \label{T4.4}
{I}(t) &= et -C\,{T}(t) \text{ and } {I}(\cdot)\text{ is nondecreasing},\\\label{T4.5}
  {I}_j(t) &\mbox{ can only increase when } {W}_j(t) =0, \, j\in \{1,...,J\},\\
\label{T4.6}
 \text{ additional} & \text{ conditions on } ({Q}(\cdot) , {T}(\cdot)), 
\,\text{ specific to the discipline.}
\end{align}
Any pair $({Q}(\cdot),{T}(\cdot))$ satisfying these equations is called a {fluid solution}. In addition, the set of all fluid solutions to the equations \eqref{T4.bal}-\eqref{T4.6} is called the {associated fluid network}\index{fluid network!associated to a queueing network}, denoted by $\mathcal{FN}$. The associated fluid network is a purely deterministic network which is based on the mean values of the primitive increments of the stochastic queueing network.
Further, as shown in \cite{dai}, for any fluid limit $\overline{X}(\cdot)$ the pair  $(\overline{Q}(\cdot),\overline{T}(\cdot))$ is a fluid solution. An immediate consequence of the above theorem is that $\mathcal{FLM} \subset \mathcal{FN}$. Note that, in general the inclusion is strict, see \cite[Section~2.7]{daischool}. Furthermore, the fluid solutions are not unique in general. A related counterexample can be found in \cite{bramsonLN} Example~1 in Section~4.3.

Let $\mathcal{Q}$ denote the set of all fluid level processes $Q(\cdot)$  such that there is an allocation process $T(\cdot)$ such that the pair $(Q(\cdot),T(\cdot))$ is a fluid solution. 
An associated fluid network is called stable if there exists a $\tau > 0$ such that
$Q(t) \equiv 0$ for all $t\geq \tau \|Q(0)\|$ and 
for all $Q \in \mathcal{Q}$.
In order to state generic properties of the set of fluid level processes $\Q$ we introduce a scaling and a shift operator. Given a function $f:\R_+ \rightarrow \R_+^K$, for $r>0$ the \emph{scaling operator} $\sigma_r$ is defined by $ \sigma_r \,f\,(t) := \tfrac{1}{r} f(r\,t)$, and for $s\geq0$ the \emph{shift operator} $\delta_s$ is defined by $ \delta_s \,f\,(t) :=  f(t+s)$. The subsequent proposition summarizes well-known generic properties of fluid networks.

\begin{proposition}\label{prop:fluids-closedGFN}
The set $\Q$ is nonempty and satisfies
\begin{enumerate}
\item[$(a)$] there is a $L>0$, such that for any $Q \in \Q$ and $t,s \in  \R_+$ it holds that 
\begin{equation*}
\|Q(t)-Q(s)\| \leq L \, |t-s|.
\end{equation*}
\item[$(b)$] $Q \in \Q$ implies $\sigma_r\,Q  \in \Q$ for all $r>0$.
\item[$(c)$] $Q \in \Q$ implies $\delta_s\,Q \in \Q$ for all $s\geq0$.
\item[$(d)$] if a sequence $(Q_n)_{n \in \N}$ in $\Q$ converges to $Q_*$ u.o.c., then $Q_* \in \Q$.
\end{enumerate}
\end{proposition}

Let $Q_1$ and $Q_2$ be two fluid level processes in $\Q$ satisfying $Q_1(t^*)=Q_2(0)$ for some $t^* \geq 0$. Then, $Q_1\diamond_{t^*}Q_2$ is called the concatenation of $Q_1$ and $Q_2$ at $t^*$, which is defined by 
\begin{equation*}
Q_1\diamond_{t^*}Q_2(t):= \begin{cases} 
Q_1(t) &\quad t \leq t^*, \\
Q_2(t-t^*) &\quad t \geq t^*.
\end{cases}
\end{equation*}
The set $\Q$ of fluid level processes is said to satisfy the \emph{concatenation property} if for any $Q_1$ and $Q_2$ in $\Q$ such that $Q_1(t^*)=Q_2(0)$ for some $t^* \geq 0$ it holds that $Q_1\diamond_{t^*}Q_2 \in \Q$. Moreover, let $\Q_q := \{ Q \in \Q: Q(0)=q\}$ denote the set of fluid level processes in $\Q$ starting in $q\in \R_+^n$. Further, the set-valued map $q \rightsquigarrow \Q_q$ is \emph{lower semicontinuous} if for each $q \in \R_+^K$, $Q \in \Q_{q}$ and $(q_n)_{n \in \N}$ converging to $q$, there is a sequence of fluid level processes $(Q_n)_{n \in \N}$ with $Q_n \in \Q_{q_n}$ which converges to $Q$ u.o.c.

\begin{remark}\label{rem:concatenation-lsc}
In \cite{PhD-schoenlein} it is shown that for fluid networks under general work-conserving, priority and HL proportional processor sharing disciplines the set $\mathcal{Q}$ of fluid level processes has the properties $(a)-(d)$ and additionally
\begin{enumerate}
\item[$(e)$] it satisfies the concatenation property,
\item[$(f)$] the set-valued map $q \rightsquigarrow \Q_{q}$ is lower semicontinuous.
\end{enumerate}
We note that $FIFO$ fluid networks the concatenation property does not hold in general, cf. Example~5.2.15 in \cite{PhD-schoenlein}.
\end{remark}

To characterize stability of associated fluid networks let $\mathcal{K}$ denote the set of all functions $w:\R_+ \to \R_+$ being strictly increasing with $w(0)=0$. Further, let $\mathcal{K}_{\infty}$ denote all unbounded class $\mathcal{K}$ functions.
%
Given an associated fluid network, a function $V : \R_+^K \rightarrow \R_+$ is said to be a \emph{Lyapunov function} if there exist class $\mathcal{K}$ functions $w_i$, $i=1,2,3$ such that
\begin{align}\label{fLF1}
w_1(\|q\|) \leq V(q)  &\leq w_2(\|q\|),\\\label{fLF2}
V(Q(t)) - V(Q(s)) &\leq - \int_{s}^{t} \,w_3(\|Q(r)\|)\,\di{r}
\end{align}
for all $0\leq s \leq t \in \R_+$ and all trajectories $Q \in \Q$.
%
A related definition was given by Dai, see \cite{daischool}.

\begin{theorem}[\cite{QS-2012}]\label{thm:converse-lyap}
An associated fluid network satisfying $(a)-(f)$ is stable if and only if it admits a continuous Lyapunov function and the  functions $w_i$, $i=1,2,3$ are of class $\K_{\infty}$.
\end{theorem}

\section{Proof of the Theorem~\ref{thm:strict-dai}}\label{sec:foster-Lyapunov}

In this section we exhibit how the converse Lyapunov Theorem~\ref{thm:converse-lyap} can be used to give a new proof of fluid approximation theorem if the associated fluid network satisfies the conditions $(a)-(f)$. Based on the Lyapunov function of the associated fluid network we define a Foster-Lyapunov function for the underlying Markov process of the HL queueing network. To this end, we prove a modified version of the Foster-Lyapunov theorem which is appropriate for our purpose. Then, we apply the modified Foster-Lyapunov Theorem to conclude the positive Harris recurrence of the underlying Markov process. 


Let $a$ be a probability measure on $(0,\infty)$ and consider the Markov process $X_a$ with transition function
\begin{align*}
T_a(x,A) = \int_0^{\infty} P(t,x,A) a(\di{t}),
\end{align*}
where $x\in \vartimes$ and $A\in \B$. Let $\mu$ be some nontrivial measure on $(E,\B(E))$. A nonempty set $A \in \B$ is called \emph{petite} if there is a nontrivial measure $\mu$ on $(\vartimes,\B)$ and a probability measure $a$ on $(0,\infty)$ such that the transition function $T_a(x,B)$ of the sample process satisfies $ T_a(x,B) \geq \mu(B)$ for all $x\in A$ and for all $B\in \B$. The subsequent modification of the Foster-Lyapunov criterion for positive Harris recurrence is close to Proposition~4.5 in \cite{bramsonLN}.

\begin{theorem}\label{generalized-foster}
Suppose that $X$ is continuous time Markov process, such that there exist $\e>0$, $\kappa >0$, $c>0$, and measurable function $W:\vartimes \rightarrow \R$ with $W(x)\geq \delta >0$ and
\begin{align}\label{cond:Foster-Lyap}
\E_x[\, W(X(c\,W(x))) \,] \leq \max\{W(x),\kappa\} - \e W(x) 
\end{align}
for all $x\in \vartimes$. Then, for all $x$,
\begin{align*}
\E_x[\,\tau_B(\delta)\,] \leq \tfrac{1}{\e}\max\{W(x),\kappa\},
\end{align*} 
where $B:=\{x\in \vartimes \, :\, W(x)\leq \kappa\}$. In particular, if $B$ is a closed petite set, $X$ is positive Harris recurrent.
\end{theorem}
\begin{proof}
Let $\kappa>0, \e>0$ and $c>0$ and let $(T_n)_{n \in \N}$ denote a sequence of stopping times defined by $T_0:=0$ and
\begin{align}\label{eq:def-T_n-foster}
 T_{n+1}  := T_n + c\, W(X(T_n)).
\end{align}
We abbreviately denote by $\F_n:= \F_{T_n}$ the $\sigma$-algebra corresponding to the stopping time $T_n$. The strong Markov property and condition \eqref{cond:Foster-Lyap} imply that
\begin{multline}\label{foster-1}
\E_x[W(X(T_n))\, |\, \F_{n-1} ] =  \E_x[\,W(\,X(\,T_{n-1} +c\,W(X(T_{n-1})\,)\,)\,) |\, \F_{n-1} ] \\
 =  \E_{X(T_{n-1})}[\,W(X(\,c\,W(X(T_{n-1})\,)\,)\,)] \\
 \leq \max\{W(X(T_{n-1})),\kappa\} - \e W(X(T_{n-1})).
\end{multline}
Further, let $M(0):= \max\{W(x),\kappa\}$ and for $n \geq 1$ we define
\begin{align}\label{foster-2}
 M(n):= c W(X(T_n)) + \e T_n.
\end{align}
Also, we note that $T_n \in \F_{n-1}$. 

Next, we show that for $N=\inf\{n \in \N \, :\, M(n) \in B \}$ and for all $n \leq N$ we have 
\begin{align*}
\E_x[\, M(n) \, |\, \F_{n-1}    \, ] \leq M(n-1).
\end{align*}
To see this, first note that for $n \leq N$ it holds that $\max\{W(X(T_{n-1})),\kappa\}= W(X(T_{n-1}))$ as well as $T_n \in \F_{n-1}$. Moreover, using \eqref{eq:def-T_n-foster} and \eqref{foster-1} it holds that
\begin{multline*}
\E_x[\, M(n) \, |\, \F_{n-1}    \, ] = \E_x[\,c W(X(T_n)) + \e T_n \, |\, \F_{n-1}    \, ] \\
= c\,\E_x[\, W(X(T_n))  \, |\, \F_{n-1}    \, ]  +\e  \E_x[\, T_n \, |\, \F_{n-1}    \, ] \\
\leq  c \max\{W(X(T_{n-1})),\kappa\} - \e W(X(T_{n-1}))  + \e  \E_x[\, T_n\, |\, \F_{n-1}    \, ] \\
= c W(X(T_{n-1})) - \e(T_n - T_{n-1}) + \e  \E_x[\, T_n\, |\, \F_{n-1}    \, ]
= M(n-1).
\end{multline*}
The validity of the last equality follows from the basic properties of stopping times and expectations. Hence, $M( \min\{n,N\})$ is a supermartingale on $\F_n$. Moreover, by the Optional Sampling Theorem we have that
\begin{align}\label{foster-3}
\E_x[\, M(N) \, ] \leq  \E_x[\, M(0) \, ] = \max\{W(x),\kappa\}.
\end{align}
Besides, since $W(X(T_n)) \leq M(n)$ it also holds that $\tau_B(\delta) \leq T_N$ and together with \eqref{foster-2}, \eqref{foster-3} it follows that for all $x \in \vartimes$ we have 
\begin{align*}
\e \,\E_x[\, \tau_B(\delta) \, ] \leq \E_x[\, M(N) \, ] \leq   \max\{W(x),\kappa\}
\end{align*} 
and, hence, we have $\prob_x[\tau_B<\infty] =1$ for all $x \in \vartimes$ and
\begin{align*}
\sup_{x \in B} \E_x[\, \tau_B(\delta) \, ] \leq \tfrac{\kappa}{\e}.
\end{align*} 
The assertion  then follows from Theorem~4.1 in \cite{bramsonLN} (or Theorem~1.1, 1.2 in \cite{tweedie1993generalized}).
\end{proof}

\begin{proof}[Proof of Theorem \ref{thm:strict-dai}]
Since $\mathcal{FLM} \subset \mathcal{FN}$ and the associated fluid network is stable, there is a $\tau>0$ such that for all $\overline{Q} \in \mathcal{FLM}$ it holds that $\overline{Q}(t)=0$ for all $t\geq \tau\|\overline{Q}(0)\|$. Also, from the converse Lyapunov Theorem~\ref{thm:converse-lyap} there is a continuous Lyapunov function $V_L$ and class $\K_{\infty}$ functions $w_i, \,i=1,2,3$ such that for $q \in \R_+^K$ we have that
\begin{align*}
w_1(\|q\|) &\leq V_L(q)  \leq w_2(\|q\|)\\
\dot V_L(\overline{Q}(t))  &\leq - \,w_3(\|\overline{Q}(t)\|).
\end{align*}
Further, let $(r_n,x_n)_{n \in \N}$ be a sequence of pairs satisfying \eqref{cond-scaling-sequence}. Then, along a subsequence, which is also indexed by $ n$, it holds that 
\begin{align*}
 \tfrac{1}{r_{n}} Q^{x_{n}} (r_{n} t) \rightarrow \overline{Q}(t) \quad \text{ u.o.c.}
\end{align*}
as $n \rightarrow \infty$. In particular, the stability of the fluid limit model implies that
\begin{align*}
 \tfrac{1}{r_{n}} Q^{x_{n}} (r_{n} \tau) \rightarrow 0.
\end{align*}
That is, for any $\tilde \e \in (0,1)$ there is a $N \in \N$ such that for all $n > N$ we have
\begin{align*}
\tfrac{1}{r_{n}} w_2^{-1}( V_L(Q^{x_{n}} (r_{n} \tau))) \leq \tilde \e.
\end{align*}
Moreover, since $w_2^{-1}( V_L(Q^{x_n}(r_n t))) \leq \| Q^{x_n}(r_n t)\|$ and by Lemma~4.5 in \cite{dai} we have that 
\begin{align*}
\left\{ \tfrac{1}{r_{n}} w_2^{-1}( V_L(Q^{x_{n}} (r_{n} \tau))), \, n \in \N \right \}
\end{align*}
is uniformly integrable. In addition, it also holds that
\begin{align*}
\lim_{n \rightarrow \infty}  \tfrac{1}{r_{n}} \| U^{x_{n}} (r_{n} \tau) \| = \lim_{n \rightarrow \infty} \tfrac{1}{r_{n}} \| V^{x_{n}} (r_{n} \tau)\| =0.
\end{align*}
The families $\left\{ \tfrac{1}{r_{n}} U_k^{x_{n}} (r_{n} \tau), \, n \in \N \right \}$ and $\left\{ \tfrac{1}{r_{n}} V_k^{x_{n}} (r_{n} \tau), \, n \in \N \right \}$ are for each $k \in \{1,...,K\}$ uniformly integrable by Lemma~4.3 in \cite{dai}. Hence, we have
\begin{align*}
\limsup_{n\rightarrow \infty} \tfrac{1}{r_{n}} \E \left[ \, w_2^{-1}( V_L( Q^{x_n}(r_n \tau)))\,+ \| U_k^{x_{n}} (r_{n} \tau)\| \, +  \| V_k^{x_{n}} (r_{n} \tau) \|\,\right] \leq \tilde \e.
\end{align*}
Thus, there is a $\kappa>0$ such that for all $r_n >\kappa$ we have
\begin{align*}
\E \left[ \, w_2^{-1}( V_L( Q^{x_n}(r_n \tau)\,)) + \|\,U^{x_{n}} (r_{n} \tau)\,\| + 
\| \, V^{x_{n}} (r_{n} \tau) \, \| \, \right] \leq r_n \,\tilde \e.
\end{align*}
We define the Foster-Lyapunov function $W: \vartimes \rightarrow \R_+$ by
\begin{align*}
W(x):= w_2^{-1} (V_L(q)) + \|u\| + \|v\|.
\end{align*}
Then, for all $x$ with $W(x) >\kappa$ it follows that
\begin{align*}
\E_{x} \left[ \, W( X( W(x)\, \tau) \,)  \, \right] \leq  \tilde \e \,W(x)
\end{align*}
and consequently,
\begin{align*}
\E_{x} \left[ \, W( X(\tau \,W(x) ) \,)  \, \right] \leq \max\{W(x),\kappa\} - \e\, W(x).
\end{align*}
Finally, we have to show that $B=\{x :W(x) \leq  \kappa\}$ is closed and petite. To see that $B$ is petite, note that by Lemma~3.1 in \cite{dai} the set $A=\{x \in \vartimes :|x|\leq \kappa\}$ is closed and petite. Furthermore, since $w_2^{-1}(V_L(q))\leq \|q\|$ it holds that $B \subset A$. In  addition, the continuity of $w_2^{-1}$ and $V_L$ imply that $B$ is closed. As subsets of petite sets are petite, the assertion then follows from Theorem~\ref{generalized-foster}.
\end{proof}

\section*{Acknowledgments}
I am grateful to Fabian Wirth for helpful discussions on fluid and queueing networks.

\end{document}